\documentclass{article}

\usepackage{amsmath,amsfonts,amssymb,url,theorem}

\usepackage[utf8]{inputenc}

{\theorembodyfont{\slshape}
\newtheorem{proposition}{Proposition}[section]
\newtheorem{lemma}[proposition]{Lemma}
\newtheorem{corollary}[proposition]{Corollary}
\newtheorem{conjecture}[proposition]{Conjecture}
\newtheorem{theorem}[proposition]{Theorem}}

{\theorembodyfont{\upshape}

}

\newcommand\Q{{\mathbb Q}}

\newcommand\Z{{\mathbb Z}}

\newcommand\qed{\hfill$\square$}

\newcommand{\lcm}{{\mathrm{lcm}\,}}

\newcommand{\ord}\nu

\newenvironment{proof}{\paragraph{Proof}}{}

\title{Number Fields in Fibers: the Geometrically Abelian Case with Rational Critical Values}

\author{Yuri Bilu\footnote{Institut de Mathématiques de Bordeaux, Université de Bordeaux \& CNRS; \textsf{yuri@math.u-bordeaux.fr}}, 
\stepcounter{footnote}
\stepcounter{footnote}
Florian Luca\footnote{School of Mathematics, Wits University, Johannesburg; \textsf{Florian.Luca@wits.ac.za}}}

1

\makeatletter

\renewcommand*\l@section[2]{%
  \ifnum \c@tocdepth >\z@
    \addpenalty\@secpenalty
    \addvspace{0.2em \@plus\p@}%
    \setlength\@tempdima{1.5em}%
    \begingroup
      \parindent \z@ \rightskip \@pnumwidth
      \parfillskip -\@pnumwidth
      \leavevmode \bfseries
      \advance\leftskip\@tempdima
      \hskip -\leftskip
      #1\nobreak\hfil \nobreak\hb@xt@\@pnumwidth{\hss #2}\par
    \endgroup
  \fi}
  
\makeatother

\begin{document}

\maketitle

\begin{abstract}
Let~$X$ be an algebraic curve over~$\Q$ and ${t\in \Q(X)}$ a non-constant rational function such that ${\Q(X)\ne \Q(t)}$. 
For every ${ n  \in \Z}$ pick ${P_ n \in X(\bar\Q)}$ such that  ${t(P_n)=n}$. We conjecture that, for large~$N$, among the number fields $\Q(P_1), \ldots, \Q(P_N)$ there are at least $cN$ distinct. We prove this conjecture in the special case when $\bar\Q(X)/\bar\Q(t)$ is an abelian field extension and the critical values of~$t$ are all rational. This implies, in particular, that our conjecture follows from a more famous conjecture of Schinzel. 
\end{abstract}

{\footnotesize 

\tableofcontents

}

\section{Introduction}

\textsl{Everywhere in this paper ``curve'' means ``smooth geometrically irreducible projective algebraic curve''. }

\bigskip

Let~$X$ be a  curve over~$\Q$  and ${t\in \Q(X)}$ a non-constant rational function such that ${\Q(X)\ne \Q(t)}$. We fix, once and for all, an algebraic closure~$\bar\Q$. All number fields occurring in this article  are subfields of this~$\bar\Q$.

Dvornicich and Zannier \cite[Theorem~2(a)]{DZ94} proved  the following theorem. 

\begin{theorem}[Dvornicich, Zannier]
\label{tdvz}
For every ${ n  \in \Z}$ pick ${P_ n \in X(\bar\Q)}$ such that  ${t(P_n)=n}$.
There exists a real number ${c>0}$ (depending on~$X$ and~$t$, but not on the particular selection of every~$P_n$) such that 
for every sufficiently large  integer~$N$ the number field ${\Q(P_1,\ldots, P_N)}$ is of degree at least $e^{cN/\log N}$ over~$\Q$.
\end{theorem}


An immediate consequence is the following result.

\begin{corollary}
\label{cdvz}
In the above set-up, there exists a real number ${c>0}$ such that for every sufficiently large integer~$N$,  among the number fields $\Q(P_1), \ldots, \Q(P_N)$ there are at least $cN/\log N$ distinct.  
\end{corollary}

Theorem~\ref{tdvz} is, in general, best possible, but Corollary~\ref{cdvz} is, probably, not; see the discussion in the introduction of~\cite{BL16}. In particular, in~\cite{BL16} we suggest the following conjecture.

\begin{conjecture}
\label{cours}
Let~$X$ be a  curve over~$\Q$  and ${t\in \Q(X)}$ a non-constant $\Q$-rational function such that ${\Q(X)\ne \Q(t)}$. Then there exists a real number ${c>0}$ such that  for every sufficiently large integer~$N$,  among the number fields $\Q(P_1), \ldots, \Q(P_N)$ there are at least $cN$ distinct. 
\end{conjecture}

There is also a more famous conjecture (attributed in~\cite{DZ94,DZ95} to Schinzel), which relates to Theorem~\ref{tdvz} in the same way as
Conjecture~\ref{cours} relates to Corollary~\ref{cdvz}. To state it, recall that ${\alpha \in \bar \Q\cup\{\infty\}}$ is called a \textsl{critical value} (or a \textsl{branch point}) of ${t\in \bar\Q(X)}$ if the rational function\footnote{We use the standard convention ${t-\infty=t^{-1}}$.} ${t-\alpha}$ has at least one multiple zero in $X(\bar \Q)$. It is well-known that any rational function ${t\in \bar \Q(X)}$ has at most finitely many critical values, and that~$t$ has at least~$2$ distinct critical values if ${\bar\Q(X)\ne \bar\Q(t)}$ (a consequence of the Riemann-Hurvitz formula). In particular, in this case~$t$ admits at least one \textsl{finite} critical value.

\begin{conjecture}[Schinzel]
\label{cschin}
In the set-up of Conjecture~\ref{cours}, assume that either~$t$ has at least one  finite critical value not belonging to~$\Q$, or the field extension ${\bar \Q(X)/\bar\Q(t)}$ is not abelian. Then there exists a real number ${c>0}$ such that for every sufficiently large integer~$N$ the number field ${\Q(P_1,\ldots, P_N)}$ is of degree at least $e^{cN}$ over~$\Q$.
\end{conjecture}

As Dvornichich and Zannier remark in~\cite{DZ94,DZ95}, the hypothesis in Conjecture~\ref{cschin} is necessary. Indeed, when all finite critical values of~$t$ belong to~$\Q$ and the field extension ${\bar \Q(X)/\bar\Q(t})$ is  abelian, it follows from Kummer's Theory that $\Q(X)$ is contained in the field of the form ${L(t, (t-\gamma_1)^{1/e_1}, \ldots, (t-\gamma_s)^{1/e_s})}$, where~$L$ is a number field,  ${\gamma_1, \ldots, \gamma_s}$ are rational numbers and ${e_1, \ldots, e_s}$ are positive integers. Now if we denote by~$A$ the maximal absolute value of the denominators and the numerators of the rational numbers ${\gamma_1, \ldots, \gamma_s}$, and set ${E=\lcm(e_1, \ldots, e_s)}$, then the number field  ${Q(P_1, \ldots, P_N)}$ is contained in the field, generated over~$L$ by the $E$th roots of prime numbers not exceeding ${AN+A}$; by the Prime Number Theorem, 
the degree of this field cannot exceed $e^{cN/\log N}$ for some ${c>0}$. 


Dvornicich and Zannier~\cite{DZ94,DZ95} obtain several results in favor of Schinzel's Conjecture. In particular, they show \cite[Theorem~2(b)]{DZ94} that it holds true if~$t$ admits a critical value of degree~$2$ or~$3$ over~$\Q$.

In~\cite{BL16} we  improve on Corollary~\ref{cdvz}, showing that $cN/\log N$ can be replaced by $N/(\log N)^{1-\eta}$ with some ${\eta>0}$.  
See the introduction of~\cite{BL16} for further relevant references. 

The purpose of the present note is to show that Conjecture~\ref{cours} holds true in the case excluded in Schinzel's conjecture. The following theorem is proved in Section~\ref{sproof}.

\begin{theorem}
\label{tmain}
Conjecture~\ref{cours} holds true when all finite critical values of~$t$ belong to~$\Q$ and the field extension ${\bar \Q(X)/\bar\Q(t})$ is  abelian. 
\end{theorem}

An immediate consequence of Theorem~\ref{tmain} is that \textsl{Conjecture~\ref{cschin} implies Conjecture~\ref{cours}}.


\paragraph{Acknowledgments}
Yuri Bilu worked on this article when he was visiting the University of Xiamen. He thanks this institution for financial support and excellent working conditions. 

We thank Elina Wojciechowska who asked the question that instigated this note. We also thank the referee for the encouraging report and for detecting some inaccuracies.

\section{Abundance of Almost Square-Free Values of Polynomials with  Rational Roots}

\label{sgaq}

Let~$S$ be a finite set of prime numbers and~$\ell$ a positive integer. We say that  ${a\in \Z}$ is \textsl{$S$-square-free} if ${\ord_p(a)\in \{0,1\}}$ for every prime ${p\notin S}$. If, in addition to this, ${\ord_p(a)\le \ell}$ for all ${p\in S}$, then we say that~$a$ is  \textsl{$(S,\ell)$-square-free}.

We say that  integers~$a$ and~$b$ are \textsl{$S$-distinct} if there exists a prime ${p\notin S}$ such that ${\ord_p(a)\ne \ord_p(b)}$, and \textsl{$S$-equal} otherwise.  

In the following lemma we collect some elementary properties of the notions just introduced.

\begin{lemma}
\label{lssquare}
Let~$S$ and~$\ell$ be as above. 
\begin{enumerate}
\item
\label{iellone}
Let $a_1,\ldots, a_k$ be  distinct $(S,\ell)$-square-free integers which are, however, all $S$-equal. Then ${k\le 2 (\ell+1)^{|S|}}$. 

\item
\label{inisom}
Let~$L$ be a number field and~$S$ a finite set of (rational) prime numbers containing all the primes ramified in~$L$. 
Let  $a,b$ be $S$-distinct  $S$-square-free  integers.  
Let ${e>1}$ be an integer whose all prime divisors belong to~$S$, and let $A,B$ be  integers satisfying 
$$
a\mid A, \quad A\mid a^{e-1},\quad b\mid B, \quad B \mid b^{e-1}.
$$ 
Then the number fields $L(A^{1/e})$ and $L(B^{1/e})$ are not isomorphic.

\item
\label{imanyfields}
Let~$L$ and~$S$ be as in part~\ref{inisom}. 
Let $a_1,\ldots, a_N$ be distinct $(S,\ell)$-square-free   integers. Let ${e>1}$ be an integer whose all prime divisors belong to~$S$, and let $A_1, \ldots, A_n$ be positive integers satisfying 
$$
a_i\mid A_i, \quad A_i\mid a_i^{e-1} \qquad (i=1, \ldots, N). 
$$  
Then among the number fields $L(A_i^{1/e})$ there  are at least $N/2(\ell+1)^{|S|}$ distinct. 

\end{enumerate}
\end{lemma}

\begin{proof}
Part~\ref{iellone} is obvious. To prove~\ref{inisom}, observe that, by the hypothesis, there exists a prime  ${p\notin S}$ such that one of the numbers 
$\ord_p(a)$, $\ord_p(b)$
is~$1$ and the other is~$0$; say, ${\ord_p(a)=1}$ and  ${\ord_p(b)=0}$. Then ${1\le\ord_p(A)\le e-1}$ and  ${\ord_p(B)=0}$, which implies that~$p$ ramifies in the field $L(A^{1/e})$ but not in $L(B^{1/e})$.  This proves~\ref{inisom}. 
Finally,~\ref{imanyfields} follows from~\ref{iellone} and~\ref{inisom}.  \qed
\end{proof}

\bigskip

In the sequel 
$$
f(T)=\alpha_dT^d+\cdots+\alpha_0=\alpha_d(T-\gamma_1)\cdots (T-\gamma_d)\in \Z[T]
$$
is a separable polynomial whose all roots $\gamma_1,\ldots, \gamma_d$ belong to~$\Q$. 
For every prime number~$p$ set
$$
\lambda_i(p)= \ord_p(f'(\gamma_i)) \quad (i=1,\ldots, d),\qquad 
\lambda(p)=\max_{1\le i\le d}\lambda_i(p).
$$
Note that, while individual $\lambda_i(p)$ may be negative, we always have ${\lambda(p)\ge 0}$, and, moreover, 
\begin{equation}
\label{elagede}
\lambda(p)\ge\delta(p), 
\end{equation}
where
${\delta(p)=\min_{1\le i\le d}\ord_p(\alpha_i)}$.  
Indeed, it follows from the Gauss Lemma that 
$$
\delta(p)= \ord_p(\alpha_d)+\sum_{i=1}^d\min\{0,\ord_p(\gamma_i)\}.
$$
Now, if, say, ${\ord_p(\gamma_1)\ge\ord_p(\gamma_i)}$ for ${i\ge 2}$ then
$$
\lambda_1(p)=\ord_p(\alpha_d)+\sum_{i=2}^d\ord_p(\gamma_1-\gamma_i)\ge  \ord_p(\alpha_d)+\sum_{i=2}^d\min\{0,\ord_p(\gamma_i)\}\ge \delta(p), 
$$
proving~\eqref{elagede}. 

We will use the following variation of Hensel's lemma. 

\begin{lemma}
\label{lhens}
Let~$n$ be an integer such that ${\ord_p(f(n))>2\lambda(p)}$. Then  there exists a unique ${j\in \{1, \ldots, d\}}$ such that ${\ord_p(n-\gamma_j)=\ord_p(f(n))-\lambda_j}$. 
\end{lemma}

\begin{proof}
We will write~$\ord(\cdot)$, $\lambda_j$,~$\lambda$ and~$\delta$ instead of~$\ord_p(\cdot)$, $\lambda_j(p)$,~$\lambda(p)$ and~$\delta(p)$.

Choose~$j$ such that ${\ord(n-\gamma_j) \ge \ord(n-\gamma_i)}$ for all ${i\ne j}$. (\textsl{A priori} this~$j$ is not uniquely defined, but in the course of the proof we will see that it actually is.) First of all, we claim that 
\begin{equation}
\label{einte}
\ord(\gamma_j)\ge 0. 
\end{equation}
Indeed, if ${\ord(\gamma_j)< 0}$ then ${\ord(n-\gamma_i)=\ord(\gamma_i)<0}$ for all ${i=1,\ldots,n}$, which implies that 
$$
\ord(f(n)) =\ord(\alpha_d)+\sum_{i=1}^d\ord(\gamma_i)=\ord(\alpha_d)+\sum_{i=1}^d\min\{0,\ord(\gamma_i)\}=\delta.
$$
Since ${\ord(f(n))>2\lambda}$, this contradicts~\eqref{elagede}. This proves~\eqref{einte}. 

We claim further that
\begin{equation}
\label{everycl}
\ord(n-\gamma_j)>\lambda_j.
\end{equation}
Indeed, our definition of~$j$ implies that 
$$
\ord(n-\gamma_i)\le \ord(\gamma_j-\gamma_i)\qquad (i\ne j).
$$
Hence 
\begin{align*}
\ord(f(n))&= \ord(\alpha_d)+\sum_{i=1}^d\ord(n-\gamma_i)\\
&\le \ord(\alpha_d)+\sum_{i\ne j}\ord(\gamma_j-\gamma_i) + \ord(n-\gamma_j)\\
&=\lambda_j+\ord(n-\gamma_j).
\end{align*}
Therefore
$$
\ord(n-\gamma_j)\ge \ord(f(n))-\lambda_j>2\lambda-\lambda_j\ge \lambda_j,
$$
which proves~\eqref{everycl}.

Since  ${\ord(f'(\gamma_j))=\lambda_j}$, inequality~\eqref{everycl} implies that
\begin{equation}
\label{ederiv}
\ord(f'(n))=\lambda_j. 
\end{equation}
Thus, we have 
${\ord(f(n))>2\lambda \ge 2\ord(f'(n))}$. 
Hensel's lemma implies that~$f$ has a unique root ${\gamma\in \Q_p}$ with the property 
$$
\ord(n-\gamma)\ge \ord(f(n))-\ord(f'(n))>2\lambda-\lambda_j\ge \lambda_j.
$$
Since the root~$\gamma_j$ has this property, we must have ${\gamma=\gamma_j}$.

To conclude the proof of the lemma, observe that the Taylor expansion 
$$
f(n)=f(\gamma_j)+f'(\gamma_j)(x-\gamma_j)+\cdots 
$$
implies  the congruence
$$
f(n)\equiv f'(\gamma_j)(n-\gamma_j)\bmod p^{2\ord(n-\gamma_j)},
$$
which, together with~\eqref{everycl}, proves that ${\ord(n-\gamma_j)=\ord(f(n))-\lambda_j}$. \qed

\end{proof}

\bigskip

For all primes~$p$ with finitely many exceptions we have 
\begin{equation}
\label{eallzero}
\lambda_i(p)=\ord_p(\gamma_i)=0 \qquad (i=1, \ldots,d). 
\end{equation}
In particular, ${\lambda(p)=0}$ for all but finitely many~$p$. We denote by~$S_0$ the finite set of primes for which~\eqref{eallzero} does not hold, and we set ${\ell_0=\max_p\lambda(p)}$. We also denote by~$U$, respectively,~$V$, the maximum of absolute values of  the numerators, respectively, denominators, of rational numbers~$\gamma_i$: if ${\gamma_i=u_i/v_i}$ with coprime ${u_i,v_i\in \Z}$ then
$$
U=\max_{1\le i\le d}|u_i|, \qquad V=\max_{1\le i\le d}|v_i|. 
$$

The following is a version of Lemma~2 from~\cite{LS06}.

\begin{lemma}
\label{lshpar}
Let~$S$ be a finite set of primes containing~$S_0$ and let~$\ell$ be an integer satisfying ${\ell\ge2\ell_0}$. Let~$P$ be the smallest prime not belonging to~$S$. Then, given an integer ${N\ge 1}$, there are at most 
\begin{equation}
\label{ezetaand}
d\left(\zeta(\ell+1-\ell_0)+\frac1{P-1}\right)N+d(VN+U)^{1/2}+d|S|
\end{equation}
positive integers ${n\le N}$ with the property 
\begin{equation}
\label{enosq}
\text{$f(n)$ is not $(S,\ell)$-square-free.}
\end{equation}
Here $\zeta(\cdot)$ is the Riemann $\zeta$-function. 
\end{lemma}

\begin{proof}
Let ${n\in \{1,\ldots, N\}}$ satisfy~\eqref{enosq}. Then we have one of the following options:
\begin{align}
\label{epins}
\ord_p(f(n))&>\ell \qquad \text{for some ${p\in S}$},\\
\label{epnotins}
\ord_p(f(n))&>1 \qquad \text{for some ${p\notin S}$}. 
\end{align}
In the case~\eqref{epins} we have ${\ord_p(f(n))>2\ell_0\ge 2\lambda(p)}$. Lemma~\ref{lhens} implies that for some root $\gamma_i$ we have 
${n\equiv \gamma_i \bmod p^{\ord_p(f(n))-\lambda_i(p)}}$. Since ${\ord_p(f(n))\ge \ell+1}$ and ${\lambda_i(p)\le \ell_0}$, this implies 
\begin{equation}
\label{emodpel}
n\equiv \gamma_i \bmod p^{\ell+1-\ell_0}. 
\end{equation}
When~$p$ and~$i$ are fixed, the number of ${n\in \{1,\ldots, N\}}$ satisfying~\eqref{emodpel} is bounded by ${N/p^{\ell+1-\ell_0}+1}$. Summing up over all ${p\in S}$ and ${i\in \{1, \ldots, d\}}$, we estimate the total number of~$n$ satisfying~\eqref{epins} as
\begin{equation}
\label{ezeta}
d\sum_{p\in S}\left(\frac N{p^{\ell+1-\ell_0}}+1\right) \le dN\sum_p\frac 1{p^{\ell+1-\ell_0}}+d|S|=d\zeta(\ell+1-\ell_0)N+d|S|.
\end{equation}

In the case~\eqref{epnotins} we have ${\lambda(p)=0}$ and  ${\ord_p(f(n))\ge 2}$. Lemma~\ref{lhens} implies that for some root $\gamma_i$ we have 
\begin{equation}
\label{emodtwo}
n\equiv \gamma_i \bmod p^2. 
\end{equation}
Since ${1\le n\le N}$, this implies ${n=\gamma_i}$ or ${p\le (VN+U)^{1/2}}$. 

When~$p$ and~$i$ are fixed, the number of ${n\in \{1,\ldots, N\}}$ satisfying~\eqref{emodtwo} is bounded by ${N/p^2+1}$. Summing up over all~$p$ satisfying  ${P\le p\le (VN+U)^{1/2}}$ and all ${i\in \{1, \ldots, d\}}$, we estimate the total number of~$n$ satisfying~\eqref{epnotins} as
\begin{align}
d\sum_{P\le p\le (VN+U)^{1/2}}\left(\frac N{p^2}+1\right) &\le dN\sum_{p\ge P}\frac 1{p^2}+d(VN+U)^{1/2} \nonumber\\
\label{ebigp}
&\le d\frac N{P-1}+d(VN+U)^{1/2}. 
\end{align}
Summing~\eqref{ezeta} and~\eqref{ebigp}, we obtain~\eqref{ezetaand}. \qed
\end{proof}

\bigskip

An immediate consequence is that, with suitably chosen~$S$ and~$\ell$, ``most'' of the values $f(n)$ are $(S,\ell)$-square-free. Here is the precise statement. 

\begin{corollary}
\label{cshpar}
There exist a finite set of primes~$S_1$ and a positive integer~$\ell_1$ (both depending only on~$f$) such that the following holds. For every ${S\supseteq S_1}$ and every ${\ell\ge \ell_1}$ there exists ${N_0=N_0(f,S)}$ such that for ${N\ge N_0}$, at most $N/2$ positive integers ${n\le N}$ satisfy~\eqref{enosq}.  
\end{corollary}

\begin{proof}
Let~$\ell_1$ be a positive integer and $P_1$ a prime number satisfying 
$$
d\zeta(\ell_1+1-\ell_0)< \frac16, \qquad \frac d{P_1-1}< \frac16.
$$
Setting ${S_1=S_0\cup\{\text{primes $p<P_1$\}}}$, the result follows. \qed. 
\end{proof}

\section{Proof of Theorem~\ref{tmain}}
\label{sproof}

We start with the special case of a superelliptic curve. 

\begin{theorem}
\label{tsuper}
Let ${F(T)\in \Q[T]}$ be a non-constant polynomial whole all roots are rational numbers,~$L$ a number field  and~$e$  a positive integer. Assume that $F(T)$ is not an $e$th power in $\bar\Q[T]$. Then there exists a positive number~$c$ such that, for  large~$N$, among the number fields 
\begin{equation}
\label{esuper}
L(F(1)^{1/e}), \ldots, L(F(N)^{1/e})
\end{equation}
there is at least $cN$ distinct. 
\end{theorem}

\begin{proof}
We may assume that the roots of~$F$ are all of multiplicity not exceeding ${e-1}$. Furthermore, multiplying~$F$ by ${a^e}$ with a suitable non-zero integer~$a$, we may assume that  ${F(T)\in\Z[T]}$. Then there exists a separable polynomial ${f(T)\in \Z[T]}$ such that ${f(T)\mid F(T)}$ and ${F(T)\mid f(T)^{e-1}}$ in the ring $\Z[T]$.

Corollary~\ref{cshpar} implies that, with suitably chosen~$S$ and~$\ell$ the following holds: for large~$N$,  at least half of the numbers 
\begin{equation}
\label{enumbers}
f(1), \ldots, f(N)
\end{equation}
are $(S,\ell)$-square-free. The polynomial~$f$ takes every value at most~$d$ times, where ${d=\deg f}$. Hence among~\eqref{enumbers} there are at least $N/2d$ distinct  $(S,\ell)$-square-free numbers. We complete the proof applying Lemma~\ref{lssquare}:\ref{imanyfields}. \qed
\end{proof}

\bigskip

Now we can prove Theorem~\ref{tmain} in full generality. Note first of all that, if ${P,Q\in X(\bar\Q)}$ and~$L$ is a number field, then ${L(P)\ne L(Q)}$ implies ${\Q(P)\ne \Q(Q)}$. Hence, it suffices to show that, for some number field~$L$, among the fields 
\begin{equation}
\label{elfi}
L(P_1), \ldots, L(P_N)
\end{equation}
there are at least $cN$ distinct. 

Now we use Kummer's theory. Since $\bar\Q(X)/\bar\Q(t)$ is a abelian extension, for some number field~$L$ we have ${L(X)=L(t,F_1(t)^{1/e_1}, \ldots, F_s(t)^{1/e_s})}$, where  ${F_i(t)\in L[t]}$, ${e_i\ge 2}$ and $F_i(t)$ is not a $e_i$th power in $\bar\Q[t]$. 

Moreover, the roots of every $F_i$ are finite critical values of~$t$, which, by the hypothesis, belong to~$\Q$. In particular, we may assume that ${F_i(t)\in \Q[t]}$. 

Pick some~$F_i$ and~$e_i$ and call them~$F$,~$e$ in the sequel. Theorem~\ref{tsuper} implies that, for large~$N$, among the fields~\eqref{esuper} there are at least $c'N$ distinct. But $L(F(n)^{1/e})$ is a subfield of $L(P_n)$ (provided~$L$ contains the $e$th roots of unity, which can be always achieved by extending~$L$). It remains to note that the fields $L(P_n)$ are of degree over~$\Q$ bounded independently of~$n$:
$$
[L(P_n):\Q]\le [L(X):\Q(t)]. 
$$
A field of degree~$r$ over~$\Q$ may have at most $c(r)$ distinct subfields. Hence, producing $c'N$ distinct subfields of the fields~\eqref{elfi} implies that among the fields~\eqref{elfi} there are at least $cN$ distinct. \qed



\end{document}